\documentclass[12pt]{paper}
\usepackage{amsmath,amssymb,amsthm,upref,graphicx,mathrsfs}
\usepackage[top=25mm, bottom=20mm, left=25mm, right=25mm]{geometry}
\usepackage{color,xcolor}
\usepackage[
  colorlinks=true,
  linkcolor=blue,
  citecolor=blue,
  urlcolor=blue]{hyperref}

\numberwithin{equation}{section}
\newtheorem{theorem}{Theorem}[section]
\newtheorem{proposition}[theorem]{Proposition}
\newtheorem{corollary}[theorem]{Corollary}
\newtheorem{lemma}[theorem]{Lemma}

\newtheorem{definition}[theorem]{Definition}

\numberwithin{equation}{section}

\begin{document}

\title{Interpolating sequences for weighted Bergman spaces on strongly pseudoconvex bounded domains}
\author{Hamzeh Keshavarzi }

\maketitle

\begin{abstract}
Let $0<p<\infty$, $\beta>-1$, and $\Omega$ be a strongly pseudoconvex bounded domain with a smooth boundary in $\mathbb{C}^n$.
We will study the interpolation problem for weighted Bergman spaces $A^p_\beta(\Omega)$.  
In the case, $1\leq p<\infty$, and $\beta> \max \{n(2p-1)-1, n(2q-1)-1\}$,  where $q$ is the conjugate exponent of $p$ (let $q=1$, for $p=1$), we show that a sequence in $\mathbb{B}_n$, the unit ball in $\mathbb{C}^n$, is interpolating for $A^p_\beta(\mathbb{B}_n)$ if and only if it is separated.\\
\textbf{MSC(2010):} primary: $47B33$, secondary: $47B38$; $32A37$\\
\textbf{Keywords:} Strongly pseudoconvex bounded domains, interpolating sequences, weighted Bergman spaces.
\end{abstract}

\section{Introduction}

In this paper, we address the interpolating problem on the weighted Bergman spaces $A^p_\beta(\Omega)$. Let us first present some notations and a brief history:

For a positive integer $n$,
let $\Omega$ be a strongly pseudoconvex bounded domain with a smooth boundary in $\mathbb{C}^n$. Two examples of these domains are $\mathbb{D}$ (the open unit disk in the complex plane) and $\mathbb{B}_n$ (the open unit ball in $\mathbb{C}^n$). See \cite{krantz} for an account of strongly pseudoconvex bounded domains.  Throughout the paper, $a$ is a fixed point in $\Omega$.
 Space $H(\Omega)$ is the class of all analytic functions on $\Omega$ and $H^\infty (\Omega)$ is the space of all bounded functions in $H(\Omega)$.
Let $\nu$ be the Lebesque measure on $\Omega$ and $\delta(z)$ be the Euclidean distance of $z$ from the boundary of $\Omega$. If $0<p<\infty$  and $\beta >-1$, the weighted Bergman space is denoted by $A^p_\beta(\Omega)$ and defined as $A^p_\beta(\Omega)=L^p(\Omega, \delta^\beta d\nu)\cap H(\Omega)$. The classical weighted Bergman space is $A^2(\Omega)=A^2_0(\Omega)$.
For $p\geq 1$, the weighted Bergman space $A^p_\beta(\Omega)$ is a Banach space and for $0<p<1$, it is a complete metric space. For more details of  Bergman spaces see \cite{duren, hedenmalm, zhu1, zhu2}.

If $\varphi$ is a holomorphic self-map of $\Omega$, then the composition operator $C_\varphi$ on $H(\Omega)$ induced by $\varphi$ is defined as $C_\varphi f=f\circ \varphi$. It is well-known that the composition operators that are induced by automorphisms on $\mathbb{B}_n$ are bounded on $A_\beta^p(\mathbb{B}_n)$. The author \cite[Corollary 3.3]{keshavarzi} proved that an automorphism composition operator $C_\varphi$ is bounded on $A_\beta^p(\Omega)$ if and only if there exists some $C>0$ such that $(1/C)\delta(z)< \delta(\varphi(z))<C \delta(z)$ for all $z\in \Omega$. Hence, $C_\varphi$ is bounded on all $A_\beta^p(\Omega)$ if and only if it is bounded on some $A_\beta^p(\Omega)$. Two good sorce for studying the composition operators are \cite{cowen, shapiro}.

Nevannlina \cite{nevanlinna} gave a necessary and sufficient condition for a sequence to be an interpolation sequence in $H^\infty(\mathbb{D})$. Carleson \cite{carleson1} with a simpler condition presented another characterization for these sequences in $H^\infty(\mathbb{D})$.
Using interpolation sequences, Carleson solved the corona conjecture on $H^\infty (\mathbb{D})$ in his celebrated paper \cite{carleson2}.
Berndtsson \cite{bern} gave a sufficient condition for  $H^\infty(\mathbb{B}_n)$-interpolating sequences.
 Seip \cite{seip} addressed this problem on  $A_\beta^p(\mathbb{D})$ and gave a complete characterization for it. However, this problem has not been solved on $A_\beta^p(\mathbb{B}_n)$. Jevtic et al. \cite{jevtic} gave some results on $A_\beta^p(\mathbb{B}_n)$.

  In section $3$, we study  the interpolating problem on the weighted Bergman spaces $A^p_\beta(\Omega)$. In the case $1\leq p<\infty$, and
  $$\beta> \max \{n(2p-1)-1, n(2q-1)-1\},$$
  where $q$ is the conjugate exponent of $p$ (let $q=1$, when $p=1$),  we show that separated sequences are interpolating for $A^p_\beta(\Omega)$.
Also, for these $\beta$ values, we prove that a sequence in $\mathbb{B}_n$ is $A^p_\beta(\mathbb{B}_n)$-interpolating if and only if it is separation.

The notation $A\lesssim B$ on a set
$S$ means that, independent of any $a\in S$, there exists some positive constant $C$ such that $A(a)\leq  CB(a)$. In this case, we say that $\{A\}$ is, up to a constant, less than or equal to $\{B\}$. Also,
 we use the notation
$A\simeq B$ on $S$ to indicate that there are some
positive constants $C$ and $D$ such that $CB(a) \leq A(a)\leq DB(a)$ for
each $a\in S$.

\section{Preliminaries}

\textbf{The Kobayashi metric.}
The infinitesimal Kobayashi metric $F_K:\Omega\times \mathbb{C}^n \rightarrow [0,\infty)$ is defined as
$$F_K(z,w)=\inf  \Big{\{} C>0: \ \exists f\in H(\mathbb{D}, \Omega) \ with \ f(0)=z, \ f^\prime(0)=\dfrac{w}{C} \Big{\}}.$$
Where $H(\mathbb{D}, \Omega)$ is the space of analytic functions from $\mathbb{D}$ to $\Omega$.
Let $\gamma:[0,1]\rightarrow \Omega$ be a $C^1$-curve. The Kobayashi length of $\gamma$ is defined as:
$$L_K(\gamma)= \int_0^1 F_K (\gamma(t), \gamma^\prime(t))dt.$$
For $z,w\in \Omega$, the Kobayashi metric function is defined as
$$\beta(z,w)=\inf \Big{ \{} L_K(\gamma); \ \gamma \ is \ C ^1-curve \ with \ \gamma(0)=z \  and \ \gamma(1)=w \Big{\}}.$$
The Kobayashi metric of a strongly pseudoconvex bounded domain is a complete metric  (see  \cite[Corollary 2.3.53]{abate}).

\textbf{Kobayashi balls.} If $0<r<1$ and $z\in \Omega$, we shall denote by $B(z,r)$ the Kobayashi ball centered at $z$ with the radius of $r$ and define it in the following way:
$$B(z,r)=\{w\in \Omega: \ \tanh \beta(z,w)<r\}.$$
Since $\beta(.,.)$ is a complete metric, the closures of Kobayashi balls are compact.
By \cite[Lemmas 2.1 and 2.2 ]{abate1}, for every $0<r<1$, we have
\begin{equation} \label{e1}
 \delta(z)^{n+1} \simeq \delta(w)^{n+1}\simeq \nu(B(w,r))\simeq \nu(B(z,r)).
\end{equation}
for all $w\in \Omega$ and $z\in B(w,r)$.

Given $z\in \Omega$, $0<r<1$, and a sequence $\{a_k\}$ in $\Omega$, we shall denote by $N(z,r,\{a_k\})$ the number of points of $\{a_k\}$ contained in $B(z,r)$.
\begin{lemma}\cite[Lemma 2.5]{abate1} \label{l1}
Let $r\in (0,1)$. Then there exists a sequence $\{a_k\}$ of points in $\Omega$ such that $\Omega=\bigcup_{k=0}^\infty B(a_k,r)$ and $\sup_{z\in \Omega} N(z,R,\{a_k\})<\infty$, where $R=\frac{1+r}{2}$. The sequence $\{a_k\}$ is called an $r$-lattice for $\Omega$.
\end{lemma}
Also, in the proof of \cite[Lemma 2.5]{abate1}, it has been shown that the balls $B(a_k,r/3)$ are disjoint.

\begin{lemma}\cite[Corollary 2.8]{abate1} \label{l2}
Given $r\in (0,1)$, set $R=\frac{1+r}{2}$. Then there exists a $C_r>0$ such that
$$\chi(z) \leq \dfrac{C_r}{\nu(B(w,r))} \int_{B(w,R)} \chi d\nu, \qquad \forall w\in \Omega, \ \forall z\in B(w,r),$$
for every non-negative plurisubharmonic function $\chi:\Omega\rightarrow \mathbb{R}^+$.
\end{lemma}

\textbf{Bergman kernel.}
Let $K(z,w)$ be the Bergman kernel for $A^2(\Omega)$. Indeed, if $f\in A^2(\Omega)$, then
$$f(z)= \int_\Omega K(z,w) f(w) d\nu(w).$$
Since $K(.,z)=\overline{K(z,.)}\in A^2(\Omega)$, we have
$$K(z,z)= \int_\Omega |K(w,z)|^2 d\nu(w)=\|K(z,.)\|_{0,2}^2.$$
Consider $r\in (0,1)$. There is a positive constant $\delta_r$ which for any $z\in \Omega$ with $\delta(z)<\delta_r$,
\begin{equation} \label{e2}
 |K(z,w)|\simeq \dfrac{1}{\delta(z)^{n+1}}\gtrsim 1, \qquad w\in B(z,r).
\end{equation}
Also, for all $z\in \Omega$,
\begin{equation} \label{e3}
|K(z,z)|\simeq \dfrac{1}{\delta(z)^{n+1}}\gtrsim 1.
\end{equation}
For the proof of the above inequalities see \cite[page 186]{range}, \cite[page 1239]{hu}, and \cite[Lemma 3.1, Lemma 3.2 and Corollary 3.3]{abate1}.
By \cite[Corollary 11 and Theorem 13]{li},
\begin{equation} \label{e4}
 \int_\Omega |K(z,w)|^p \delta(w)^\alpha d\nu(w)\lesssim \delta(z)^{\alpha+n+1-(n+1)p},
\end{equation}
for all $z\in \Omega$, $\frac{n}{n+1}<p<\infty$ and $-1<\alpha<(n+1)(p-1)$.

\textbf{Carleson measures}.
 A positive Borel measure
$\mu$ on $\Omega$ will be called a $\beta$-Bergman Carleson measure if there exists a
positive constant $K$ such that
$$\int_{\Omega}|f(z)|^p d\mu \leq K\int_{\Omega}|f (z)|^p \delta(z)^\beta d\nu(z),$$
for all $f$ in $A^p_\beta(\Omega)$. According to the \cite{abate1} and \cite{hu}, $\mu$ is a $\beta$-Bergman Carleson measure if and only if for some (every) $0<r<1$, there exists some $C_r>0$, where $\mu(B(z,r))\leq C_r \delta(z)^{\beta+n+1}$, for all $z\in \Omega$.

\section{\textbf{Interpolating sequences for $A^p_\beta(\Omega)$}}
In this section, we study the interpolating problem on $A^p_\beta(\Omega)$ and give several results about it.
For $p>0$ and $\beta\in \mathbb{R}$, let
$$\ell_\beta^p=\ell_\beta^p(\{a_k\}):= \{v=\{v_k\}\subset \mathbb{C}: \ \{\delta(a_k)^\beta v_k\}\in \ell^p\},$$
Also, for $v\in \ell_\beta^p$,
$$\|v\|^p_{p,\beta}:= \sum_{k=1}^\infty \Big( \delta(a_k)^\beta |v_k|\Big)^p.$$

\begin{definition}
We say that $\{a_k\}$ is an interpolating sequence for $A_\beta^p(\Omega)$ or an $A_\beta^p(\Omega)$-interpolating sequence, denoted by $\{a_k\}\in Int(A_\beta^p(\Omega))$ if for any $\{v_k\}\in \ell^p_{(n+1+\beta)/p}$, there exists $f\in A_\beta^p(\Omega)$ such that $f(a_k)=v_k$, for all $k$.
\end{definition}

\begin{definition}
We say that a sequence $\{a_k\}$ is separated if there exists $\delta>0$ such that, for all $j\neq k$, $\beta(a_k,a_j)>\delta$.
\end{definition}

\begin{proposition}\label{p1}
For $\beta>-1$ and $p>0$, if $\{a_k\}\in Int(A_\beta^p(\Omega))$ is separated, then  there is a constant $M>0$ (which is called the interpolation constant of $\{a_k\}$) such that $\|f\|_{p,\beta} \leq M \|v\|_{p,(n+1+\beta)/p}$.
\end{proposition}
\begin{proof}
If we prove that $f\mapsto \{f(a_k)\}$ is a bounded operator then open mapping theorem implies the desire result. By using the assumption mentioned above, for some $\delta>0$, the balls $B(a_k,\delta)$ are pairwise disjoint. Thus,
\begin{eqnarray*}
\|f\|^p_{p,\beta} &\geq& \sum \int_{B(a_k,\delta)} |f(z)|^p \delta(z)^\beta dv(z).
\end{eqnarray*}
Using the plurisubhormonicity of $|f|^p$, we have
$$\|f\|^p_{p,\beta} \gtrsim \sum \delta(a_k)^{n+1+\beta} |f(a_k)|^p =\|\{f(a_k)\}\|^p_{(n+1+\beta)/p}.$$
the proof is complete.
\end{proof}

It is well-known that the  Bergman metric on $\mathbb{B}_n$ and the Kobayashi metric on $\Omega$ are automorphism invariant.
Hence, it is not far from mind that interpolating sequences are automorphism invariant (that is, if $\{a_k\}$ is an $A_\beta^p(\Omega)$-interpolating sequence, then so is $\{\varphi(a_k)\}$, where $\varphi$ is an automorphism). 

\begin{theorem}
Let $\varphi$ be an automorphism of $\Omega$ and the composition operator induced by it be bounded on $A_\beta^p(\Omega)$. If $\{a_k\}$ is an $A_\beta^p(\Omega)$-interpolating sequence, then so is $\{\varphi(a_k)\}$.
\end{theorem}
\begin{proof}
We define
$$T_\varphi f(z)=  \Big( K(\varphi^{-1}(a),z)\Big)^\frac{2(n+1+\beta)}{(n+1)p} f\circ \varphi(z),$$
and
$$S_\varphi f(z)= \Big( K(\varphi^{-1}(a),\varphi^{-1}(z))\Big)^\frac{-2(n+1+\beta)}{(n+1)p} f\circ \varphi^{-1}(z).$$
We show that $T_\varphi$ is invertible with $S_\varphi$ as its inverse. For any $f$ in $A_\beta^p(\Omega)$, we have
$$\|T_\varphi f\|^p_{p,\beta}= \int_\Omega |K(\varphi^{-1}(a),z)|^\frac{2(n+1+\beta)}{n+1} |f\circ \varphi(z)|^p \delta(z)^\beta d\nu(z)$$
By \ref{e5}, we have
$$ \|T_\varphi f\|^p_{p,\beta}= \int_\Omega |J_\varphi(\varphi^{-1}(a)) J_\varphi(z)|^2 |K(a,\varphi(z))|^2 |f\circ \varphi(z)|^p \delta(z)^\beta d\nu(z)$$
Since the function $K(a,z)$ has upper and lower bounds,
$$ \|T_\varphi f\|^p_{p,\beta}\simeq \int_\Omega |J_\varphi(z)|^2  |f\circ \varphi(z)|^p \delta(z)^\beta d\nu(z)$$
Now by \ref{e6},
$$ \|T_\varphi f\|^p_{p,\beta}\simeq\int_\Omega  |f(z)|^p \delta(\varphi^{-1}(z))^\beta d\nu(z).$$
Finally, \cite[Corollary 3.3]{keshavarzi} implies that
$$ \|T_\varphi f\|^p_{p,\beta} \simeq \int_\Omega  |f(z)|^p \delta(z)^\beta d\nu(z)=\|f\|^p_{p,\beta}.$$
Also, we can see that $T_\varphi S_\varphi =S_\varphi T_\varphi=I$. Therefore, by the open mapping theorem, $T_\varphi$ is invertible. Take $v\in \ell^p_{(n+1+p)/p}(\{\varphi(a_k)\})$. Then
$$\sum_{k=1}^\infty | K(\varphi^{-1}(a),a_k)|^\frac{2(n+1+\beta)}{n+1} \delta(a_k)^{n+1+\beta} |v_k|^p$$
is, up to a constant, less than or equal to
$$\sum_{k=1}^\infty  \delta(a_k)^{n+1+\beta} |v_k|^p< \infty.$$
Thus, there is an $F$ such that $\|F\|_{\beta,p}\lesssim \|v\|$ and
$$F(a_k)= K(\varphi^{-1}(a),a_k)^\frac{2(n+1+\beta)}{(n+1)p} v_k.$$
Then if $G=S_\varphi F$, we can see that $G(\varphi(a_k))=v_k$. The proof is complete.
\end{proof}

According to the following theorem, if we add a finite number of points to a $A_\beta^p(\Omega)$-interpolating separated sequence, then it remains $A_\beta^p(\Omega)$-interpolating.

\begin{theorem} \label{tc}
Let $1\leq p<\infty$ and $\beta>-1$.
The union of a $A_\beta^p(\Omega)$-interpolating separated sequence and a finite number of points is again $A_\beta^p(\Omega)$-interpolating.
\end{theorem}
\begin{proof}
It is enough to show that the union of a $A_\beta^p(\Omega)$-interpolating separated sequence and one point is $A_\beta^p(\Omega)$-interpolating. Let $\{a_k\}$ be an $A_\beta^p(\Omega)$-interpolating separated sequence and $b\in \Omega \diagdown \{a_k\}$. Hence, there is some $\delta>0$ such that $\beta(a_k,b)>\delta$, for all $k$.

We first claim that it is enough to find $f\in A_\beta^p(\Omega)$ with $f(a_k)=0$ for all $k$ and $f(b)\neq 0$. To see this, let $\{v_k\}\cup \{v_0\} \in \ell^p_{(n+1+\beta)/p}(\{a_k\}\cup\{b\})$ and let $g\in A_\beta^p(\Omega)$ be such that $g(a_k)=v_k$. Then the function
$$ F(z)=g(z) +\dfrac{v_0-g(b)}{f(b)} f(z)$$
belongs to $A_\beta^p(\Omega)$, and $F(a_k)=v_k$ for all $k$ and $F(b)=v_0$.

Suppose that all $f\in A_\beta^p(\Omega)$ with $f(a_k)=0$ for all $k$ have $f(b)=0$. This implies that for any $f\in A_\beta^p(\Omega)$, the value $f(b)$ is determined by the values of $f(a_k)$ since the difference between two functions with the same values on $\{a_k\}$ vanishes at $b$.

We define the functional $\Lambda:\ell^p \rightarrow \mathbb{C}$ by
$$\Lambda(\{\lambda_k\})=f(b),$$
where $f\in A_\beta^p(\Omega)$ is such that $f(a_k)=\delta(a_k)^{-\frac{n+1+\beta}{p}} \lambda_k$ for all $k$. 
Because $\{a_k\}$ is separated, as in Proposition \ref{p1}, one can see that $\{f(a_k)\delta(a_k)^{\frac{n+1+\beta}{p}}\}\in \ell^p$, for every $f\in A_\beta^p(\Omega)$. We will show that $\Lambda$ is identically $0$, which implies that $f(b)=0$, for all $f\in A_\beta^p(\Omega)$. This is a clear contradiction.

Since $f(b)$ is determined by the values $f(a_k)$ which are actually independent of $f$, $\Lambda$ is linear. It is also continuous:
$$\Lambda(\{\lambda_k\})=f(b)\lesssim \|f\|_{p,\beta}\leq M\|\{\delta(a_k)^{-\frac{n+1+\beta}{p}} \lambda_k\|_{p, \frac{n+1+\beta}{p}} \lesssim M\|\{\lambda_k\}\|_p$$
where $M$ denotes the interpolation constant of $\{a_k\}$. Thus, if $\frac{1}{p}+\frac{1}{q}=1$, then $\Lambda\in \ell^q$, in the sense that there exists some $\{c_k\}\in \ell^q$ such that
$$\Lambda(\lambda)=\sum \lambda_k c_k, \qquad \forall \lambda \in \ell^p.$$
Consider $\delta_{jk}$ as the Kronecker symbol  and the  function $f_j\in A_\beta^p(\Omega)$ with
$$f_j(a_k)=\delta(a_k)^{-\frac{n+1+\beta}{p}}\delta_{jk}.$$
Using the definition we have:
$$\Lambda(v^j)=f_j(b)=\sum \delta_{jk}c_k=c_j.$$
Now take the functions $F_j(z)=[|a_j-b|^2-\overline{(a_j-b)}(z-b)]f_j(z).$
Obviously, $F_j\in A_\beta^p(\Omega)$ and $F_j(a_k)=0$ for all $k$. Therefore,
$$0=F_j(b)=|a_j-b|^2f_j(b)=|a_j-b|^2c_j,$$
and hence $c_j=0$. This, shows that $\Lambda\equiv 0.$
\end{proof}

 Let $\{a_k\}$ be a sequence in $\Omega$. We define the values
$$K(\{a_k\},p,q)= \sup_{k\in \mathbb{N}} \sum_{j=1}^\infty \delta(a_k)^p \delta(a_j)^q| K(a_k,a_j)|^\frac{p+q}{n+1}.$$
Based on the above values of $K(\{a_k\},p,q)$, we give sufficient conditions (Theorems \ref{ta} and \ref{tb}) for a sequence in $\Omega$ to be $A^p_\beta (\Omega)$-interpolating. First, we need the following lemma.
The equivalence of parts (i) and (ii) below has been given in \cite[Theorem 4.2]{abate1}.

\begin{lemma} \label{l5}
The following conditions are equivalent.
\begin{itemize}
\item[(i)] $\Gamma=\{a_k\}$ is the union of a finite number of separated sequences.
\item[(ii)] $\sup_{z\in \Omega} N(z,r,\{a_k\})<\infty$ for all (some) $r\in (0,1)$.
\item[(iii)] $K(\{a_k\},p,q)<\infty$, for all $q>n$ and all $p\leq q$.
\item[(iv)] There exist $q> n,p$ such that  $K(\{a_k\},p,q)<\infty$.
\item[(v)] for all $p>0$ and $q>n$
$$ \sup_{z\in \Omega} \sum_{j=1}^\infty \delta(z)^p \delta(a_j)^q |K(z,a_j)|^\frac{p+q}{n+1}<\infty.$$
\item[(vi)] $ \sum_{j=1}^\infty \delta(a_j)^q \delta_{a_j} $ is a $(q-n-1)$-Carleson measure for $q>n$.
\end{itemize}
\end{lemma}
\begin{proof}

(ii)$\Rightarrow$ (v): Let $0<r<1$, the  plurisubhormonicity of $|K(z,a_j)|^\frac{p+q}{n+1}$ and \ref{e1} imply that
$$\sum_{a_j\in \Gamma} \delta(z)^p  \delta(a_j)^q |K(z,a_j)|^\frac{p+q}{n+1}\lesssim
\sum_{a_j\in \Gamma} \delta(z)^p \int_{B(a_j,r)}  |K(z,w)|^\frac{p+q}{n+1} \delta(w)^{q-n-1} d\nu(w)$$
Statement (ii) tells us that the above value is, up to a constant, less than or equal to
$$\delta(z)^p \int_{\Omega}  |K(z,w)|^\frac{p+q}{n+1} \delta(w)^{q-n-1} d\nu(w)$$
Now, \ref{e4} gives the desire result.

(v)$\Rightarrow$ (iii) $\Rightarrow$ (iv) are abvious.

(iv) $\Rightarrow$ (vi)
Let $\delta(z)>\delta_r$. Then there is some $0<s<1$ such that $B(z,r)\subset \{w: \ \delta(w)\geq s\}$  by the compactness of $ \{w: \ \delta(w)\geq s\}$;
$$\dfrac{\mu(B(z,r))}{\delta(z)^{q}} \leq \dfrac{\mu(\{w: \ \delta(w)\geq s\})}{\delta(z)^{q}}\lesssim 1.$$
Now Let $\delta(z)<\delta_r$. Let $z\in B(a_k,r)$ for some $k$ (otherwise $\mu(B(z,r))=0$). Then
$$\mu(B(z,r))=\sum_{a_j\in \Gamma\cap B(z,r)} \delta(a_j)^q.$$
 Using the triangle inequality and the monotonicity of $\tanh$, if, $w\in B(a_k,r)$ and $r^\prime=\tanh (2\tanh^{-1} (r))$, then $B(w,r)\subseteq B(a_k,r^\prime)$. Thus, we can let $\Gamma\cap B(z,r)\subset B(a_k,r^\prime)$. Therefore, by \ref{e2},
\begin{eqnarray*}
\mu(B(z,r)) &\simeq& \sum_{a_j\in \Gamma \cap B(z,r)} \delta(a_j)^q \delta(a_k)^p |K(a_k,a_j)|^\frac{p}{n+1}\\
&\simeq& \delta(z)^q \sum_{a_j\in \Gamma \cap B(z,r)} \delta(a_j)^q \delta(a_k)^p |K(a_k,a_j)|^\frac{p+q}{n+1} \lesssim \delta(z)^q.
\end{eqnarray*}

(vi) $\Rightarrow$ (ii) If $\delta(a_k)\nrightarrow 0$, then there is a subsequence $\{a_{n_k}\}$ of $\{a_k\}$ and $\delta>0$ such that $\delta(a_{n_k})>\delta$, for all $k\in \mathbb{N}$. Thus, for $f\equiv 1$;
$$\int_\Omega |f|^p d\mu =\sum \delta (a_k)^q\geq \sum \delta (a_{n_k})^q =\infty.$$
This contradicts (v). Hence, $\sup_{\delta(z)>\delta_r} N(z,r,\Gamma)<\infty$.
Now let $\delta(z)<\delta_r$, hence, by \ref{e2},
$$N(z_0,r,\Gamma) \lesssim \sum_{z\in \Gamma \cap B(z_0,r)} |K(z_0,z)|^2$$
Inequality \ref{e1}, the plurisubhormonicity of $|K(z,a_j)|^2$, and \ref{e4} imply that
\begin{eqnarray*}
N(z_0,r,\Gamma) &\lesssim&  \delta(z_0)^{-q} \sum_{z\in \Gamma \cap B(z_0,r)} \delta(z)^q |K(z_0,z)|^2\\
&\lesssim& \delta(z_0)^{-q} \int_\Omega   |K(z_0,w)|^2 \delta(w)^{q-n-1} d\nu(w)\\
&\lesssim& \delta(z_0)^{-q} \delta(z_0)^{q-n-1+n+1-2(n+1)}=1.
\end{eqnarray*}
\end{proof}

\begin{theorem}\label{ta}
Let $\{a_k\}$ be a union of a finite number of separated sequences in $\Omega$. If there exists an $m>0$ such that $K(\{a_k\},m,n+1+\beta)<1$, then $\{a_k\}$ is $A^1_\beta (\Omega)$-interpolating.
\end{theorem}

\begin{proof}
 Consider the restriction map $T: A_\beta^p(\Omega)\rightarrow \ell _{\frac{n+1+\beta}{p}}^p$, where
$T(f)=\{f(a_k)\}$. Hence,
$$\|T(f)\|_{\frac{n+1+\beta}{p}}^p=\sum |f(a_k)|^p \delta(a_k)^{n+1+\beta}=\int_\Omega |f|^p d\mu. $$
Where $\mu=\sum \delta(a_k)^{n+1+\beta} \delta_{a_k}$. Since $\{a_k\}$ is a finite union of separated sequences, by Lemma \ref{l5}, $\mu$ is a $\beta$-Carleson measure. Therefore, $T$ is a bounded operator.
 In order to show that $T$ is onto, we define, given $v=\{v_k\}\in \ell _{n+1+\beta}^1(\{a_k\})$, the approximation extension
\begin{equation*}
E(v)(z)=\sum v_k \delta(a_k)^{n+1+\beta+m} K(a_k,z)^\frac{n+1+\beta+m}{n+1}.
\end{equation*}
Using \ref{e4}, it is straightforward that 
\begin{eqnarray*}
\|E(v)\|_{1,\beta} &\leq& \int_\Omega \delta(z)^\beta \sum |v_k| \delta(a_k)^{n+1+\beta+m} |K(a_k,z)|^\frac{n+1+\beta+m}{n+1}d\nu(z)\\
&=&   \sum |v_k| \delta(a_k)^{n+1+\beta+m} \int_\Omega |K(a_k,z)|^\frac{n+1+\beta+m}{n+1}\delta(z)^\beta d\nu(z)\\
&\lesssim& \sum |v_k|\delta(a_k)^{n+1+\beta}=\|v\|_{\ell _{n+1+\beta}^1(\{a_k\})}.
\end{eqnarray*}
Hence, $E(v)$ is in $A_\beta^1(\Omega)$. On the other hand, $TE-I$,  regarded as an operator on $\ell _{n+1+\beta}^1(\{a_k\})$, has norm strictly smaller than $1$:
\begin{eqnarray*}
\|TE(v)-v\|&=& \sum \delta(a_k)^{n+1+\beta} |TE(v)_k-v_k|\\
&\leq& \sum \delta(a_k)^{n+1+\beta} \sum_{j: j\neq k} |v_j| \delta(a_j)^{n+1+\beta+m} | K(a_k,a_j)|^{\frac{n+1+\beta+m}{n+1}}\\
&=&  \sum_{j=1}^\infty  \delta(a_j)^{n+1+\beta} |v_j|  \sum_{k: k\neq j}\delta(a_j)^{m} \delta(a_k)^{n+1+\beta} | K(a_k,a_j)|^{\frac{n+1+\beta+m}{n+1}}\\
&\leq& C\|v\|.
\end{eqnarray*}
Where, $C=K(\{a_k\},m,n+1+\beta)<1$.  This shows that $T$ is onto and the proof is complete.
\end{proof}

\begin{theorem} \label{tb}
Let $\{a_k\}$ be a union of a finite number of separated sequences in $\Omega$, $1<p<\infty$, and $q$ be its conjugate exponent. If there exist $c_1,c_2>0$ such that $c_1c_2<1$ and
\begin{equation*}
K\Big(\{a_k\},\dfrac{n+1+\beta}{p},\dfrac{n+1+\beta}{q}\Big)<c_1^q, \qquad K\Big(\{a_k\},\dfrac{n+1+\beta}{q},\dfrac{n+1+\beta}{p} \Big)<c_2^p,
\end{equation*}
then $\{a_k\}$ is $A_\beta^p(\Omega)$-interpolating.
\end{theorem}
\begin{proof}
Given $\{v_k\}\in \ell_{\frac{n+1+\beta}{p}}^p(\{a_k\})$, take the approximate extention
\begin{equation*}
E(v)(z)=\sum v_k \delta(a_k)^{n+1+\beta} K(z,a_k)^\frac{n+1+\beta}{n+1} .
\end{equation*}
Using the duality, the reproducing kernel for $A_\beta^p(\Omega)$, and Lemma \ref{l5} with $\mu=\sum \delta(a_k)^{n+1+\beta}\delta_{a_k}$, one has $\|E(v)\|\leq C\|v\|$.

On the other hand, if $T$ denotes the operator on $A_\beta^p(\Omega)$ associated with $\{a_k\}$, we have $\|TE-I\|<1$, since
\begin{eqnarray*}
\|TE(v)-v\|&=& \sum \delta(a_k)^{n+1+\beta} |TE(v)_k-v_k|^p\\
&\leq& \sum_{k=1}^\infty \Big(  \sum_{j: j\neq k} \delta(a_k)^\frac{n+1+\beta}{p}\delta(a_j)^\frac{n+1+\beta}{q} | K(a_k,a_j)|^{\frac{n+1+\beta}{n+1}}\Big)^\frac{p}{q}\\
&\times&  \sum_{j: j\neq k} \delta(a_k)^\frac{n+1+\beta}{p}\delta(a_j)^\frac{n+1+\beta}{q} | K(a_k,a_j)|^{\frac{n+1+\beta}{n+1}} \delta(a_j)^{n+1+\beta}|v_j|^p\\
&\leq& c_1^p  \sum_{j=1}^\infty  \delta(a_j)^{n+1+\beta} |v_j|^p  \sum_{k: k\neq j}\delta(a_j)^\frac{n+1+\beta}{p} \delta(a_k)^\frac{n+1+\beta}{p} | K(a_k,a_j)|^\frac{n+1+\beta}{n+1}\\
&\leq& (c_1c_2)^p \|v\|.
\end{eqnarray*}
This shows that $TE$ is invertible and so $T$ is onto.
\end{proof}

\begin{corollary} \label{c1}
Let $1\leq p<\infty$, and 
 $\{a_k\}$ be a union of a finite number of separated sequences in $\Omega$. Then there is a positive constant $D$ such that if $\delta(a_k)\leq \frac{1}{2^k D}$, then $\{a_k\}$ is an interpolating sequence for $A_\beta^p(\Omega)$.
\end{corollary}
\begin{proof}
We prove only the case $p=1$. The case $p>1$ is obtained similarly.
Let $C$ be the constant that
$$|K(z,z)|\leq C \delta(z)^{-(n+1)}, \qquad \forall z\in \Omega.$$
 The following calculations show that if $D$ is such that $C^\frac{m+n+\beta+1}{n+1} /D<1$, then the preceding theorem implies that our claim holds.
\begin{eqnarray*}
K(\{a_k\},m,n+1+\beta) &=& \sup \sum_{j=1}^\infty \delta(a_k)^m \delta(a_j)^{n+\beta+1} |K(a_k,a_j)|^\frac{m+n+\beta+1}{n+1}\\
 &\leq& C^\frac{m+n+\beta+1}{n+1} \sup \sum_{j=1}^\infty \delta(a_k)^m \delta(a_j)^{n+\beta+1} \delta(a_k)^{-m/2} \delta(a_j)^{-(n+\beta+1)/2}\\
&=&C^\frac{m+n+\beta+1}{n+1} \sup \sum_{j=1}^\infty \delta(a_k)^{m/2} \delta(a_j)^{(n+\beta+1)/2}\\
&\leq& C^\frac{m+n+\beta+1}{n+1}/D  \sum_{j=1}^\infty (\dfrac{1}{2})^j=C^\frac{m+n+\beta+1}{n+1} /D.
\end{eqnarray*}
Theorem \ref{ta} gives the result.
\end{proof}

\begin{theorem} \label{te}
Let $1\leq p<\infty$, $\{a_k\}$ be a union of a finite number of separated sequences in $\Omega$ and
\begin{itemize}
\item if $p=1$, then $\beta>n-1$, and
\item if $p>1$, then $\beta> \max \{n(2p-1)-1, n(2q-1)-1\}$,  where $q$ is the conjugate exponent of $p$.
\end{itemize}
 Then $\{a_k\}$ is an $A_\beta^p(\Omega)$-interpolation sequence. Moreover, if $\{a_k\}$ is a $r$-lattice sequence, then it is an $A_\beta^p(\Omega)$-interpolation sequence.
\end{theorem}
\begin{proof}
One can see that there exists some $C>1$ such that
\begin{equation*}
K\Big(\{a_k\},\dfrac{n+1+\beta}{q},\dfrac{n+1+\beta}{p}\Big)\leq C \sum \delta(a_j)^{(n+1+\beta)/2p}
\end{equation*}
By \ref{e1} and Lemma \ref{l5}, we have
\begin{eqnarray*}
\sum \delta(a_j)^{(n+1+\beta)/2p}&\lesssim& \sum \int_{B(a_j,r)}  \delta(z)^{(\beta+n+1)/2p-(n+1)} d\nu(z)\\
&\lesssim& \int_{\Omega} \delta(z)^{(\beta+n+1)/2p-(n+1)} d\nu(z).
\end{eqnarray*}
Since $\beta> \max \{n(2p-1)-1, n(2q-1)-1\},$ the above integral is finite.
Thus, there is some $N_1\in \mathbb{N}$ such that
\begin{equation*}
\sum_{j=N_1}^\infty \delta(a_j)^{(n+1+\beta)/2p}\leq \dfrac{1}{2C}
\end{equation*}
Hence,
$$K\Big(\{a_k\}_{k=N_1}^\infty,\dfrac{n+1+\beta}{p},\dfrac{n+1+\beta}{q}\Big)<1.$$
Similarly, we can find an $N_2$ such that
$$K\Big(\{a_k\}_{k=N_2}^\infty,\dfrac{n+1+\beta}{q},\dfrac{n+1+\beta}{p}\Big)<1.$$
If $N=\max \{N_1,N_2\}$, Theorem \ref{tb} implies that $\{a_k\}_{k=N}^\infty$ is an $A_\beta^p(\Omega)$-interpolating sequence. Finally, Theorem \ref{tc} completes the proof.
\end{proof}

Jevtic et al. \cite[Lemma 1.5]{jevtic} proved that an interpolating sequence for  $A^p_\beta(\mathbb{B}_n)$ is separated. Above theorem says that for suitable $\beta$, separated sequences are $A^p_\beta(\Omega)$-interpolating. Thus, the folowing result is straightforward.

\begin{corollary}\label{tf}
Let $1\leq p<\infty$, $\{a_k\}$ be a sequence in $\mathbb{B}_n$, and
 \begin{itemize}
\item if $p=1$, then $\beta>n-1$, and
\item if $p>1$, then $\beta> \max \{n(2p-1)-1, n(2q-1)-1\}$,  where $q$ is the conjugate exponent of $p$.
\end{itemize}
 Then $\{a_k\}$ is an $A_\beta^p(\mathbb{B}_n)$-interpolating sequence if and only if it is separated.
\end{corollary}

\section{Some basic results}

The following lemma is well-known in the unit ball. Indeed, the subharmonicity of the analytic functions on the unit ball implies it. However, subharmonicity does not necessarily hold on $\Omega$. Hence, we give a short and new proof for this lemma.
\begin{lemma}\label{l3}
Let $0<p<\infty$ and $0<R<r$. Then there is a positive constant $C$ such that for all $f\in H(\Omega)$:
$$|\nabla f (z)|\leq C \Big( \int_{B(a,r)} |f(w)|^p d\nu (w) \Big)^\frac{1}{p}, \qquad z\in B(a,R).$$
\end{lemma}
\begin{proof}
Let $L_z$ be the point evaluation at $z$  in $A^p(B(a,r))$. Then
$$\dfrac{\partial}{\partial z_k} f(z)= \lim_{\lambda\rightarrow 0} \dfrac{f(z+\lambda e_k)-f(z)}{\lambda}=\langle f, \lim_{\lambda\rightarrow 0} \dfrac{L_{z+\lambda e_k}-L_z}{\lambda}\rangle.$$
We put
\begin{eqnarray*}
\lim_{\lambda\rightarrow 0} \dfrac{K_{z+\lambda e_k}-K_z}{\lambda}= S_{z_k,z}
\end{eqnarray*}
Also, there is a positive constant $C$ such that
$$\| S_{z_k,z} \|\leq C, \qquad z\in B(a,R), \ 1\leq k\leq n.$$
Thus, for $1\leq k\leq n$,
$$|\dfrac{\partial}{\partial z_k} f(z)|=| \langle f, S_{z_k,z} \rangle | \leq C \Big( \int_{B(a,r)} |f(w)|^p d\nu (w) \Big)^\frac{1}{p}.$$
This completes the proof.
\end{proof}

Let $\varphi$ be an automorphism of $\Omega$. We use $J_\varphi(z)$ to denote the complex Jacobian of $\varphi$ at $z$. If $J_{\varphi,R}(z)$ is the real Jacobian of $\varphi$, then it is well-known that
$$J_{\varphi,R}(z)=|J_\varphi(z)|^2,$$
and
\begin{equation} \label{e5}
 K(z,w)=J_\varphi(z) \overline{J_\varphi(w)} K(\varphi(z),\varphi(w)),
\end{equation}
for all $z,w\in \Omega$. The above equations imply a useful relationship:
\begin{equation} \label{e6}
\int_\Omega f(z) d\nu(z)
=\int_\Omega f(\varphi(z)) |J_{\varphi}(z)|^2 d\nu(z)
=\int_\Omega f(\varphi(z)) \dfrac{K(z,z)}{K(\varphi(z),\varphi(z))} d\nu(z)
\end{equation}

\begin{lemma} \label{l5}
Let $u\in \Omega$ such that there is an automorphism $\varphi_u$ on $\omega$ where $\varphi_u(a)=u$ and $C_{\varphi_u}$ is bounded on $A^p_\beta(\Omega)$. Then there is a positive constant $C$ such that
$$ C^{-1} \leq \Big| \dfrac{K(z,u)}{K(z,v)} \Big| \leq C  ,$$
for all $z, v\in \Omega$ with $\beta(u,v)\leq R$.
\end{lemma}
\begin{proof}
By \ref{e5}, we have
\begin{eqnarray*}
I(z)&=&\Big| \dfrac{K(z,u)}{K(z,v)}\Big| = \Big| \dfrac{J_{\varphi_{u}^{-1}}(z)\overline{J_{\varphi_{u}^{-1}}(u)}K(\varphi_{u}^{-1}(z),\varphi_{u}^{-1}(u))}{J_{\varphi_{u}^{-1}}(z)\overline{J_{\varphi_{u}^{-1}}(v)}K(\varphi_{u}^{-1}(z),\varphi_{u}^{-1}(v))}\Big| \\
&=& \dfrac{|J_{\varphi_{u}^{-1}}(u)|}{|J_{\varphi_{u}^{-1}}(v)|} \times \dfrac{|K(\varphi_{u}^{-1}(z),a)|}{|K(\varphi_{u}^{-1}(z),\varphi_{u}^{-1}(v))|}.
\end{eqnarray*}
For each $z\in \Omega$, we have
$$|K(\varphi_{u}^{-1}(z),a)|\leq \|K(. ,a)\|_{\infty}  <\infty .$$
Thus, again by \ref{e5}
$$I(z)\lesssim \dfrac{|J_{\varphi_{u}^{-1}}(u)|}{|J_{\varphi_{u}^{-1}}(v)|}  =\dfrac{\|K_u\| \|K_{\varphi_{u}^{-1}(v)}\| }{\|K_v\| \|K_{a}\| }\lesssim \|K_{\varphi_{u}^{-1}(v)}\|.$$
Finaly, since $\varphi^{-1}_u(v)\in B(a,r)$, we conclude the desired result.
\end{proof}

\begin{lemma}
Let $0<r<1$ and $\{a_k\}$ be an $r$-lattice for $\Omega$. Then, for each $k\geq 1$, there exists a Borel set $D_k$ satisfying the following conditions:
\begin{itemize}
\item[(i)] $D(a_k,r/3) \subset D_k \subset D(a_k,r)$, for every $k$.
\item[(ii)] $D_k\cap D_j=\emptyset$, for $k\neq j$.
\item[(iii)] $\Omega=\cup D_k$.
\end{itemize}
\end{lemma}
\begin{proof}
For every $k\geq 1$, let
$$E_k=D(a_k,r)-\cup_{j\neq k} D(a_j,r/3).$$
Then $D(a_k,r/3) \subset E_k \subset D(a_k,r)$ (since $D(a_k,r/3)$ are disjoint). Also, $\{E_k\}$ covers $\Omega$: if $z\in D(a_j,r/3)$, for some $j\neq k$, then $z\in E_j$; otherwise $z\in E_k$.

 Let $D_1=E_1$ and  $D_{k+1}=E_{k+1} \setminus \cup_{i=1}^k D_i$. $\{D_k\}$ is a disjoint cover of $\Omega$: In fact, if $z\in \Omega$, then $z\in E_k$ for some $k$. If $k=1$, then $z\in D_1$. If $k>1$, then we either have $z\in D_i$ for some $1\leq i\leq k$, or $z \notin \cup_{i=1}^{k-1} D_i$. Thus $z\in E_k \setminus \cup_{i=1}^{k-1} D_i=D_k$.

 For each $k\geq 1$, we have $D_k\subset E_k \subset D(a_k,r)$. To see that each $D_{k+1}$ contains $D(a_{k+1},r/3)$ ($D_1=E_1\supset D(a_1,r/3)$), we fix $k\geq 1$, and fix $z\in D(a_{k+1},r/3)\subset E_{k+1}$. Then $z\notin E_i$, for any $i\neq k+1$ which implies that $z\notin D_i$ for any  $i\neq k+1$. This shows that
$$z\in E_{k+1} - \cup _{i=1}^k = D_{k+1}.$$
\end{proof}

\vspace*{1cm}
 Hamzeh Keshavarzi

E-mail: Hamzehkeshavarzi67@gmail.com

Department of Mathematics, College of Sciences,
Shiraz University, Shiraz, Iran

\end{document}